\documentclass[11pt,reqno]{amsart}

\usepackage{amsmath,amssymb,amsthm,mathtools,mathrsfs}
\usepackage[T1]{fontenc}
\usepackage{lmodern}
\usepackage{microtype}
\usepackage[colorlinks=true,linkcolor=blue,citecolor=blue,urlcolor=blue]{hyperref}
\usepackage{enumitem}

\allowdisplaybreaks
\newcommand{\E}{\mathop{\textsf E}\nolimits}
\newcommand{\Pp}{\mathop{\textsf P}\nolimits}
\newcommand{\Var}{\operatorname{Var}}
\newcommand{\Cov}{\operatorname{Cov}}
\newcommand{\Tr}{\operatorname{Tr}}
\newcommand{\HS}{\mathrm{HS}}
\newcommand{\Law}{\mathcal L}
\newcommand{\cJ}{\mathcal J}
\newcommand{\cE}{\mathcal E}
\newcommand{\dd}{\,\mathrm d}

\newcommand{\norm}[1]{\left\lVert #1\right\rVert}
\newcommand{\ip}[2]{\left\langle #1,#2\right\rangle}
\newcommand{\R}{\mathbb R}

\newcommand{\Id}{I}

\theoremstyle{plain}
\newtheorem{theorem}{Theorem}[section]
\newtheorem{lemma}[theorem]{Lemma}
\newtheorem{proposition}[theorem]{Proposition}
\newtheorem{corollary}[theorem]{Corollary}
\theoremstyle{definition}

\theoremstyle{remark}
\newtheorem{remark}[theorem]{Remark}

\numberwithin{equation}{section}

\title[The optimal rate in the averaged random-marginal CLT]{The Optimal Rate in the Averaged Random-Marginal Central Limit Theorem for Log-Concave Measures}
\author{Xuanang Hu}
\address{Shandong University, Jinan, Shandong, China}
\date{}

\hypersetup{
 pdftitle={The Optimal Rate in the Averaged Random-Marginal Central Limit Theorem for Log-Concave Measures},
 pdfauthor={Xuanang Hu},
 pdfsubject={Random one-dimensional marginals of isotropic log-concave measures},
 pdfkeywords={log-concave measure, random marginal, Wasserstein distance, quadratic variance inequality, optimal rate}
}


\begin{document}

\begin{abstract}
Let $X$ be a centered isotropic log-concave random vector in $\R^n$. For
$\theta\in S^{n-1}$, let $\mu_\theta$ be the law of $\ip{X}{\theta}$, and let
$\Theta$ be uniformly distributed on $S^{n-1}$, independently of $X$. We
prove the sharp estimate
\[
 \E W_1(\mu_\Theta,\gamma_1)
 \le \frac{C}{n}.
\]
Here the Wasserstein distance is computed after the direction is fixed and is
then averaged over the sphere. No symmetry assumption is imposed. A product
measure with centered exponential coordinates gives a matching lower bound
of order $n^{-1}$.

The proof separates the averaged-direction law from the fluctuation among
fixed directions. For the first part, a Taylor expansion in the random radius
retains a mean-zero cancellation and yields an $O(n^{-1})$ error. For the
second, a weighted $L^2$ distance between distribution functions is converted
into an exact spherical kernel depending only on $|x|^2$, $|y|^2$, and
$\ip{x}{y}$. Expanding this kernel in $\ip{x}{y}$, we control its linear,
quadratic, and cubic terms using the quadratic variance inequality
$\Var(X^TMX)\le 8\Tr(M^2)$, while fixed-order moment estimates control the remainder.
\end{abstract}

\maketitle

\section{Introduction}

Let $X$ be a centered isotropic log-concave random vector in $\R^n$. Thus
\[
 \E X=0,
 \qquad
 \E XX^T=I_n.
\]
Since the covariance is nonsingular, the law of $X$ has a density of the
form $e^{-\mathcal V}$ on a convex set, where $\mathcal V$ is convex. We
write $\gamma_1$ for the standard Gaussian law on $\R$, and
$\sigma_{n-1}$ for the normalized rotation-invariant probability measure on
$S^{n-1}$. For probability measures $\mu$ and
$\nu$ on $\R$ with finite first moments, let $\Pi(\mu,\nu)$ denote the set
of probability measures on $\R^2$ with marginals $\mu$ and $\nu$, and define
\begin{equation}\label{eq:w1-definition}
 W_1(\mu,\nu)
 :=\inf_{\pi\in\Pi(\mu,\nu)}
 \int_{\R^2}|x-y|\,\dd\pi(x,y).
\end{equation}
Equivalently, the infimum may be taken over all pairs of random variables
$(U,V)$ with laws $\mu$ and $\nu$, in which case the cost is $\E|U-V|$.
For $\theta\in S^{n-1}$, write
\[
 \mu_\theta:=\Law_X\ip{X}{\theta}.
\]
Thus the direction is held fixed in $\mu_\theta$, and only $X$ is random.
Throughout, $\E$ and $\Pp$ denote expectation and probability. We omit a
subscript when the source of randomness is clear from the argument and use
one only when it prevents ambiguity. The central limit theorem for
log-concave measures asks whether the one-dimensional marginals $\mu_\theta$
are approximately Gaussian for most directions when $n$ is large.

The first geometric approach related Gaussian marginals to concentration of
the Euclidean norm.  In the symmetric-convex-body setting, suitable
concentration of $|X|$ was shown to imply Gaussian marginals in most
directions; see \cite[Theorem~4]{AnttilaBallPerissinaki}.  The qualitative
central limit theorem for general isotropic log-concave measures was proved
in total variation in \cite[Theorem~1.1]{KlartagCLT}.  It was followed by
power-law quantitative estimates in \cite[Theorem~1.4]{KlartagPower} and by
pointwise estimates for random higher-dimensional marginals in
\cite[Theorem~1]{EldanKlartag}.  A common feature of these results is that
radial concentration supplies the Gaussian reference law, while additional
arguments are needed to show that most fixed directions are close to that
reference law.

A sharper quantitative behavior is possible in randomized formulations. For
sums of independent variables, \cite[Theorem~1.1]{KlartagSodin} obtains an
$O(n^{-1})$ bound, in the interval metric, for most coefficient vectors under
a bounded average fourth moment. For dependent isotropic vectors, the
second-order correlation method of \cite[Theorem~1.1]{BCG2019} combines
concentration of smooth functions on the sphere with a normal-approximation
inequality. Write $F_\theta$ for the distribution function of $\mu_\theta$,
$\Phi$ for the standard Gaussian distribution function, and
$d_K(F,G)=\sup_t|F(t)-G(t)|$. In the symmetric case that method yields
\[
 \E d_K(F_\Theta,\Phi)
 \le C\Lambda\frac{\log n}{n};
\]
here $\Lambda$ is the second-order correlation constant; see also
\cite[Theorem~17.1]{BCG2023}. Without symmetry,
\cite[Theorem~17.2]{BCG2023} gives, under a spectral-gap assumption with
constant $\lambda_1$, the averaged bound
\[
 \E\int_\R
 \bigl(F_\Theta(t)-\Phi(t)\bigr)^2\,\dd t
 \le \frac{C}{\lambda_1^2n^2}.
\]
These results show that second-order information can improve the classical
$n^{-1/2}$ scale.  They do not, however, give the sharp averaged $W_1$ bound
for general, possibly non-symmetric, log-concave measures. The related
Wasserstein normal-approximation and random-projection literature is discussed
separately in Remark~\ref{rem:transport-clt}.

There are two sources of error in the quantity studied here.  If $X$ and
$\Theta$ are averaged simultaneously, the law of $\ip{X}{\Theta}$ depends
only on $|X|$; this is the radial part of the problem.  In contrast,
\[
 \E W_1(\mu_\Theta,\gamma_1)
\]
first fixes the direction, computes the distance, and only then averages.
It therefore also contains the fluctuation of the marginal law as the
direction changes.  Estimating the radial deviation by
$\E||X|^2-n|/n$ loses a square root and gives only $O(n^{-1/2})$.  The proof
below keeps the cancellation in the radial variable and treats the
directional fluctuation through one exact integral formula rather than by
estimating angular degrees separately.

The thin-shell conjecture was recently proved with a universal constant in
\cite[Theorem~1.1]{KlartagLehec}. Two contemporaneous works then identified
the sharp radial constant. Theorem~1.1 of \cite{ChenKlartag} proves
\[
 \Var(|X|^2)\le 8n,
\]
and Theorem~1.2 there proves the sharp third-moment tensor estimate
$\|\E X^{\otimes3}\|_{\mathrm{HS}}^2\le4n$.  On the other hand,
\cite[Theorem~1.2]{Letwin} proves the stronger quadratic-form inequality
\[
 \Var(X^TMX)\le 2\E|\nabla(X^TMX)|^2
 =8\Tr(M^2),
 \qquad M=M^T.
\]
The last equality uses $\nabla(X^TMX)=2MX$ and isotropy.  We use this full
matrix statement, rather than only its radial specialization.  Besides
$\Var(|X|^2)$, it controls a vector, a matrix, and a third-order tensor that
occur in the first three terms of the kernel expansion.  In particular, the
sharp thin-shell estimate and the sharp third-moment estimate from
\cite{ChenKlartag} are important parallel results, but they do not by
themselves give the matrix estimate required in the quadratic term below.

\begin{theorem}\label{thm:main}
There is a universal constant $C$ such that, for every $n\ge1$ and every
centered isotropic log-concave random vector $X\in\R^n$,
\begin{equation}\label{eq:main}
 \E W_1(\mu_\Theta,\gamma_1)
 \le \frac{C}{n}.
\end{equation}
\end{theorem}

The order in Theorem~\ref{thm:main} cannot be improved for all isotropic
log-concave laws.

\begin{theorem}\label{thm:lower}
There is a universal constant $c>0$ such that, for every sufficiently large
$n$, one can find a centered isotropic log-concave random vector
$X^{(n)}\in\R^n$ satisfying
\begin{equation}\label{eq:lower-main}
 \E W_1\bigl(\Law_{X^{(n)}}\ip{X^{(n)}}{\Theta},\gamma_1\bigr)
 \ge \frac{c}{n}.
\end{equation}
One may take $X^{(n)}$ to have independent coordinates, each distributed as
$E-1$, where $E$ has the exponential distribution with mean one.
\end{theorem}

\begin{remark}[Wasserstein normal approximation and random projections]
\label{rem:transport-clt}
The $W_1$ distance is a standard metric in quantitative normal approximation.
For a systematic treatment by Stein's method, including independent sums and
several dependence structures, see \cite{ChenGoldsteinShao}. For sums of
independent real random variables, Rio obtained transport-distance bounds for
$1\le p\le2$ in \cite[Theorem~4.1 and Corollary~4.2]{Rio}; Bobkov extended
the corresponding Berry--Esseen inequality to every $p\ge1$ in
\cite[Theorem~1.1]{BobkovTransport}. Extensions to locally dependent random
variables are developed in \cite{LiuAustern}.

Closer to the present geometric setting, Bobkov studied Sudakov's typical
distributions in the Kantorovich--Rubinstein distance, equivalently $W_1$, in
\cite{BobkovSudakov}. Meckes treated random projections in
\cite{MeckesProjection}. A distinction in the latter work is important here:
its Wasserstein estimate concerns the law obtained after the vector and the
direction are averaged together, whereas its fixed-direction results are
formulated in the bounded-Lipschitz distance; see
\cite[Theorems~4--6]{MeckesProjection}. Our quantity
$\E W_1(\mu_\Theta,\gamma_1)$ instead computes $W_1$ for each fixed direction
before averaging over the sphere. Thus the cited results provide the relevant
Wasserstein and random-projection background, but do not yield
Theorem~\ref{thm:main} or its sharp $n^{-1}$ rate.
\end{remark}

\begin{corollary}\label{cor:typical}
For every $0<\delta<1$, all but a set of spherical measure at most $\delta$
of directions satisfy
\[
 W_1(\mu_\theta,\gamma_1)\le \frac{C}{\delta n}.
\]
\end{corollary}

\begin{proof}
Set $D(\theta)=W_1(\mu_\theta,\gamma_1)$. Theorem~\ref{thm:main} gives
$\E D(\Theta)\le C/n$. Markov's inequality says that, for every
$a>0$,
\[
 \Pp\{D(\Theta)>a\}\le\frac{\E D(\Theta)}{a}.
\]
Taking $a=C/(\delta n)$ gives
\[
 \Pp\left\{D(\Theta)>\frac{C}{\delta n}\right\}
 \le \frac{\delta n}{C}\E D(\Theta)
 \le\delta.
\]
\end{proof}

We summarize the new points of the argument. Let
\[
 \bar\mu:=\int_{S^{n-1}}\mu_\theta\,\dd\sigma_{n-1}(\theta)
\]
be the law obtained after averaging the direction. First, we prove
\[
 W_1(\bar\mu,\gamma_1)\le \frac Cn.
\]
The key point is that the linear term in the random scale has mean zero; one
must use this cancellation before taking absolute values.

Second, for the variation in the direction we introduce
\[
 \cJ(\mu,\nu)
 =\int_\R(1+t^2)(F_\mu(t)-F_\nu(t))^2\,\dd t.
\]
The change of variables $\Psi(t)=t+t^3/3$ turns this integral into an exact
expectation involving $|\Psi(u)-\Psi(v)|$.  Spherical averaging then produces
an explicit kernel depending only on
\[
 |x|^2,\qquad |y|^2,\qquad \ip{x}{y}.
\]
We compare the original pair with a pair having the same radii and independent
uniform directions.  The constant term cancels because the radii are the
same.  The linear, quadratic, and cubic terms are controlled, respectively,
by
\[
 \E[(|X|^2-n)X],\qquad
 \E[(|X|^2-n)XX^T],\qquad
 \E X^{\otimes3}.
\]
The quadratic-form inequality bounds all three objects at exactly the scale
needed for an $O(n^{-2})$ bound on the squared distribution-function error.
After taking a square root, this gives the desired $O(n^{-1})$ Wasserstein
bound.  Finally, an explicit product measure gives a matching lower bound.
Thus the proof both identifies the optimal rate and explains why the full
quadratic-form estimate, rather than a radial thin-shell estimate alone, is
the relevant input.

\begin{remark}\label{rem:conjectures}
Theorem~\ref{thm:main} does not prove the Kannan--Lov\'asz--Simonovits
conjecture. It uses the quadratic-form estimate
\cite[Theorem~1.2]{Letwin}, while the KLS conjecture asks for a
dimension-free Poincar\'e inequality for every locally Lipschitz function.
The same external work proves $\psi_n\le C\log^{1/4}n$ in
\cite[Theorem~1.1]{Letwin}.

Taking $M=I_n$ in the quadratic-form estimate gives
\[
 \Var(|X|^2)\le8n.
\]
This sharp estimate is also the conclusion of
\cite[Theorem~1.1]{ChenKlartag}.  Moreover,
\cite[Theorem~1.2]{ChenKlartag} gives the sharper tensor bound
$\|\E X^{\otimes3}\|_{\mathrm{HS}}^2\le4n$.  We retain the short deductions
from the full quadratic-form inequality because the same input is also
needed for the matrix $B$ in Lemma~\ref{lem:B}.  Indeed,
\[
 \bigl||X|-\sqrt n\bigr|
 =\frac{\bigl||X|^2-n\bigr|}{|X|+\sqrt n}
 \le\frac{\bigl||X|^2-n\bigr|}{\sqrt n}.
\]
After squaring and taking expectation, we obtain
\[
 \E(|X|-\sqrt n)^2
 \le \frac1n\E(|X|^2-n)^2
 \le8.
\]
Thus a dimension-free thin-shell bound is already part of the external
input. The new conclusion here is the sharp averaged $W_1$ rate for random
one-dimensional projections. The argument does not remove the logarithm
from the known averaged Kolmogorov estimate, does not give an $n^{-1}$ bound
for every fixed direction, and does not imply the full KLS conjecture.
\end{remark}

Section~\ref{sec:basic} records the one-dimensional and spherical
formulas used throughout, and Section~\ref{sec:lower} proves sharpness.  We
then collect the external estimates and prove the radial part of the upper
bound.  The final two sections derive the exact spherical kernel and estimate
its Taylor expansion.  Algebraic differentiations are placed in the
appendix.

\section{Two basic formulas}\label{sec:basic}

We record two standard formulas that will be used several times.

\begin{lemma}\label{lem:w1-line}
Let $\mu$ and $\nu$ be probability measures on $\R$ with finite first
moments, and let
\[
 Q_\mu(s)=\inf\{t:F_\mu(t)\ge s\},
 \qquad 0<s<1,
\]
with the analogous definition for $Q_\nu$. Then
\begin{equation}\label{eq:w1-cdf}
 W_1(\mu,\nu)
 =\int_0^1|Q_\mu(s)-Q_\nu(s)|\,\dd s
 =\int_\R|F_\mu(t)-F_\nu(t)|\,\dd t.
\end{equation}
Moreover, if $h:\R\to\R$ is one-Lipschitz, then
\begin{equation}\label{eq:w1-test}
 \left|\int h\,\dd\mu-\int h\,\dd\nu\right|
 \le W_1(\mu,\nu).
\end{equation}
\end{lemma}

\begin{proof}
The first equality in \eqref{eq:w1-cdf} is the one-dimensional monotone
transport formula \cite[Theorem~2.18]{VillaniTopics}.  The second equality
is the main result of \cite{Vallender}.

For \eqref{eq:w1-test}, let $\pi\in\Pi(\mu,\nu)$. Since $h$ is
one-Lipschitz,
\[
 \left|\int h\,\dd\mu-\int h\,\dd\nu\right|
 =\left|\int_{\R^2}(h(x)-h(y))\,\dd\pi(x,y)\right|
 \le\int_{\R^2}|x-y|\,\dd\pi(x,y).
\]
Taking the infimum over $\pi$ proves the claim.
\end{proof}

\begin{lemma}\label{lem:sphere-basic}
Let $n\ge2$, let $G=(G_1,\ldots,G_n)$ have independent standard Gaussian
coordinates, and put $U=G/|G|$. Then $U$ is uniform on $S^{n-1}$ and is
independent of $|G|$. If $k_1,\ldots,k_r$ are nonnegative integers and
$K=k_1+\cdots+k_r$, then
\begin{equation}\label{eq:sphere-even-moments}
 \E\prod_{j=1}^rU_j^{2k_j}
 =\frac{\prod_{j=1}^r(2k_j-1)!!}
 {n(n+2)\cdots(n+2K-2)}.
\end{equation}
Every moment containing an odd power of one coordinate is zero. Also,
$U_1$ has density
\begin{equation}\label{eq:U1-density}
 c_n(1-u^2)^{(n-3)/2}\mathbf1_{[-1,1]}(u),
 \qquad
 c_n=\frac{\Gamma(n/2)}{\sqrt\pi\,\Gamma((n-1)/2)}.
\end{equation}
\end{lemma}

\begin{proof}
The proof of \cite[Theorem~2]{VignatBhatnagar} uses the Gaussian polar
decomposition $G=|G|U$, in which $U$ is uniform on the sphere and is
independent of $|G|$; that theorem also gives the spherical moment formula
used below. We only write the coordinate calculation needed
here.

Since $G=|G|U$ and $|G|$ is independent of $U$,
\[
 \E\prod_{j=1}^rG_j^{2k_j}
 =\E|G|^{2K}\,\E\prod_{j=1}^rU_j^{2k_j}.
\]
The left-hand side is $\prod_j(2k_j-1)!!$. Since $|G|^2$ has the
$\chi_n^2$ distribution,
\[
 \E|G|^{2K}
 =2^K\frac{\Gamma(n/2+K)}{\Gamma(n/2)}
 =n(n+2)\cdots(n+2K-2).
\]
This proves \eqref{eq:sphere-even-moments}. Odd moments vanish by invariance
under a sign change of any coordinate.

Finally,
\[
 U_1^2=\frac{G_1^2}{G_1^2+\sum_{j=2}^nG_j^2}.
\]
The two terms in the denominator are independent chi-square variables with
one and $n-1$ degrees of freedom. Thus $U_1^2$ has the beta distribution
with parameters $1/2$ and $(n-1)/2$, and the change of variables
$u\mapsto u^2$ gives \eqref{eq:U1-density}.
\end{proof}

\section{The matching lower bound}\label{sec:lower}

Let $E$ be exponential with mean one and put $\xi=E-1$. Then $\xi$ is
centered, has variance one, and has a log-concave density. For a real random
variable $Y$, write
\[
 \varphi_Y(t):=\E e^{itY},\qquad t\in\R,
\]
for its characteristic function. Let
$X^{(n)}=(\xi_1,\ldots,\xi_n)$ with independent coordinates. For
$\theta\in S^{n-1}$ set
\[
 S_\theta=\sum_{j=1}^n\theta_j\xi_j,
 \qquad
 \alpha_3(\theta)=\sum_{j=1}^n\theta_j^3,
 \qquad
 \beta_4(\theta)=\sum_{j=1}^n\theta_j^4.
\]

\begin{lemma}\label{lem:cf-lower}
There are universal constants $t_0,c_0,C_0>0$ such that, for
$0<t\le t_0$ and every $\theta\in S^{n-1}$,
\begin{equation}\label{eq:cf-lower}
 W_1(\Law(S_\theta),\gamma_1)
 \ge c_0t^2|\alpha_3(\theta)|-C_0t^3\beta_4(\theta)-C_0t^8|\alpha_3(\theta)|^3.
\end{equation}
\end{lemma}

\begin{proof}
The characteristic function of $\xi$ is
\[
 \varphi_\xi(s)
 =\E e^{is(E-1)}
 =e^{-is}\int_0^\infty e^{-(1-is)x}\,\dd x
 =\frac{e^{-is}}{1-is}.
\]
Fix $|s|\le1/2$.  Then $1-is$ belongs to the open right half-plane, so the
principal logarithm is defined at $1-is$.  Moreover, since $|is|<1$,
\[
 -\operatorname{Log}(1-is)=\sum_{k=1}^\infty\frac{(is)^k}{k}.
\]
Let $L(s)$ be the analytic logarithm of $\varphi_\xi(s)$ near the origin
which satisfies $L(0)=0$.  The preceding identity gives
\begin{align}
 L(s)
 &=-is-\operatorname{Log}(1-is)\notag\\
 &=-is+\sum_{k=1}^\infty\frac{(is)^k}{k}\notag\\
 &=-\frac{s^2}{2}-i\frac{s^3}{3}+\rho(s),
 \label{eq:log-cf-exp}
\end{align}
where
\[
 \rho(s):=\sum_{k=4}^\infty\frac{(is)^k}{k}.
\]
The remainder is bounded explicitly by
\begin{equation}\label{eq:rho-bound}
 |\rho(s)|
 \le\sum_{k=4}^\infty\frac{|s|^k}{k}
 \le\frac{|s|^4}{4(1-|s|)}
 \le\frac12|s|^4.
\end{equation}

Choose $t_0\le1/2$.  Since $|\theta_j|\le1$, the expansion above may be
applied to every $s=t\theta_j$ whenever $0<t\le t_0$.  Independence of the
coordinates gives
\begin{align*}
 \varphi_{S_\theta}(t)
 &=\prod_{j=1}^n\varphi_\xi(t\theta_j)\\
 &=\exp\left(\sum_{j=1}^nL(t\theta_j)\right).
\end{align*}
Using $\sum_j\theta_j^2=1$ and the definitions of $\alpha_3(\theta)$ and
$\beta_4(\theta)$, we obtain
\begin{equation}\label{eq:cf-Stheta-exp}
 \varphi_{S_\theta}(t)
 =\exp\left(
 -\frac{t^2}{2}
 -i\frac{t^3}{3}\alpha_3(\theta)
 +\varepsilon_\theta(t)
 \right),
\end{equation}
where
\[
 \varepsilon_\theta(t):=\sum_{j=1}^n\rho(t\theta_j).
\]
By \eqref{eq:rho-bound},
\begin{equation}\label{eq:eps-theta-bound}
 |\varepsilon_\theta(t)|
 \le\frac12t^4\sum_{j=1}^n\theta_j^4
 =\frac12t^4\beta_4(\theta).
\end{equation}
Also,
\[
 \beta_4(\theta)=\sum_{j=1}^n\theta_j^4
 \le\left(\max_{1\le j\le n}\theta_j^2\right)
     \sum_{j=1}^n\theta_j^2
 \le1,
\]
and
\begin{equation}\label{eq:A-absolute-bound}
 |\alpha_3(\theta)|
 \le\sum_{j=1}^n|\theta_j|^3
 \le\left(\max_{1\le j\le n}|\theta_j|\right)
     \sum_{j=1}^n\theta_j^2
 \le1.
\end{equation}

Set
\[
 u:=\frac{t^3}{3}\alpha_3(\theta),
 \qquad
 \varepsilon:=\varepsilon_\theta(t).
\]
Equation \eqref{eq:cf-Stheta-exp} becomes
\[
 \varphi_{S_\theta}(t)=e^{-t^2/2}e^{-iu}e^{\varepsilon}.
\]
From the power series of the exponential,
\[
 |e^{\varepsilon}-1|
 \le\sum_{m=1}^\infty\frac{|\varepsilon|^m}{m!}
 =e^{|\varepsilon|}-1
 \le e^{|\varepsilon|}|\varepsilon|.
\]
Choose $t_0$ small enough that $t_0^4/2\le1$.  Then
\eqref{eq:eps-theta-bound} implies $|\varepsilon|\le1$, and hence $e^{|\varepsilon|}\le e$.  Thus
\begin{equation}\label{eq:exp-eps-bound}
 |e^{\varepsilon}-1|
 \le e|\varepsilon|
 \le\frac e2t^4\beta_4(\theta).
\end{equation}
Since
\[
 e^{-iu}e^{\varepsilon}=e^{-iu}+e^{-iu}(e^{\varepsilon}-1),
\]
the triangle inequality implies
\begin{align*}
 |\operatorname{Im}\varphi_{S_\theta}(t)|
 &=e^{-t^2/2}\left|\operatorname{Im}(e^{-iu}e^{\varepsilon})\right|\\
 &\ge e^{-t^2/2}|\operatorname{Im}(e^{-iu})|
      -e^{-t^2/2}|e^{\varepsilon}-1|\\
 &=e^{-t^2/2}|\sin u|-e^{-t^2/2}|e^{\varepsilon}-1|.
\end{align*}
For $0<t\le t_0$,
\[
 e^{-t^2/2}\ge e^{-t_0^2/2}=:c_1>0.
\]
Combining this bound with \eqref{eq:exp-eps-bound} gives
\begin{equation}\label{eq:imag-before-sine}
 |\operatorname{Im}\varphi_{S_\theta}(t)|
 \ge c_1|\sin u|-Ct^4\beta_4(\theta).
\end{equation}
To estimate the sine term, write
\[
 \sin u-u=\int_0^u(\cos v-1)\,\dd v.
\]
Since
\[
 |1-\cos v|
 =\left|\int_0^v\sin w\,\dd w\right|
 \le\int_0^{|v|}w\,\dd w
 =\frac{v^2}{2},
\]
we obtain
\[
 |\sin u-u|
 \le\int_0^{|u|}\frac{v^2}{2}\,\dd v
 =\frac{|u|^3}{6}.
\]
Hence
\[
 |\sin u|
 \ge |u|-\frac{|u|^3}{6}.
\]
Using $u=t^3\alpha_3(\theta)/3$ in \eqref{eq:imag-before-sine}, we obtain
\begin{equation}\label{eq:imag-final-bound}
 |\operatorname{Im}\varphi_{S_\theta}(t)|
 \ge ct^3|\alpha_3(\theta)|
     -Ct^4\beta_4(\theta)
     -Ct^9|\alpha_3(\theta)|^3.
\end{equation}

It remains to convert \eqref{eq:imag-final-bound} into a lower bound for
$W_1$.  Define
\[
 h_t(x):=\frac{\sin(tx)}{t}.
\]
Since $h_t'(x)=\cos(tx)$, we have $|h_t'(x)|\le1$, so $h_t$ is
one-Lipschitz.  For every real random variable $Y$,
\begin{align*}
 \varphi_Y(t)
 &=\E\bigl(\cos(tY)+i\sin(tY)\bigr),\\
 \operatorname{Im}\varphi_Y(t)
 &=\E\sin(tY),
\end{align*}
and therefore
\begin{equation}\label{eq:ht-characteristic}
 \E h_t(Y)=\frac{1}{t}\operatorname{Im}\varphi_Y(t).
\end{equation}
If $N$ is standard Gaussian, then
$\varphi_N(t)=e^{-t^2/2}$ is real, so
\[
 \E h_t(N)=\frac{1}{t}\operatorname{Im}\varphi_N(t)=0.
\]
Applying \eqref{eq:w1-test} from Lemma~\ref{lem:w1-line} to the
one-Lipschitz function $h_t$ yields
\begin{align*}
 W_1(\Law(S_\theta),\gamma_1)
 &\ge\left|\E h_t(S_\theta)-\E h_t(N)\right|\\
 &=\frac{|\operatorname{Im}\varphi_{S_\theta}(t)|}{t}.
\end{align*}
Substitution of \eqref{eq:imag-final-bound} proves \eqref{eq:cf-lower}.
\end{proof}

\begin{lemma}\label{lem:sphere-lower}
For $\Theta$ uniform on $S^{n-1}$,
\begin{align}
 \E A(\Theta)^2
 &=\frac{15}{(n+2)(n+4)},\label{eq:A2}\\
 \E A(\Theta)^4
 &=\frac{135(5n+72)}{(n+2)(n+4)(n+6)(n+8)(n+10)},\label{eq:A4}\\
 \E B(\Theta)&=\frac{3}{n+2}.
 \label{eq:Bmean}
\end{align}
In particular,
\begin{equation}\label{eq:A1-lower}
 \E|A(\Theta)|\ge\frac{c}{n},
 \qquad
 \E|A(\Theta)|^3\le\frac{C}{n^3}.
\end{equation}
\end{lemma}

\begin{proof}
We first explain which terms in the expansions can have nonzero expectation.
The uniform distribution on $S^{n-1}$ is invariant under reflection in any
coordinate.  Thus, for every $r\in\{1,\ldots,n\}$,
\[
 (\Theta_1,\ldots,\Theta_r,\ldots,\Theta_n)
 \stackrel{d}{=}
 (\Theta_1,\ldots,-\Theta_r,\ldots,\Theta_n).
\]
It follows that
\begin{equation}\label{eq:odd-coordinate-zero}
 \E\prod_{j=1}^n\Theta_j^{m_j}=0
 \quad\text{whenever at least one exponent $m_j$ is odd.}
\end{equation}
Indeed, reflecting a coordinate with odd exponent changes the sign of the
integrand but does not change its distribution.

Expanding the square gives
\[
 A(\Theta)^2
 =\sum_{i=1}^n\Theta_i^6
  +\sum_{i\ne j}\Theta_i^3\Theta_j^3.
\]
Every term in the second sum has an odd power of both $\Theta_i$ and
$\Theta_j$, and hence has expectation zero by
\eqref{eq:odd-coordinate-zero}.  By exchangeability of the coordinates and
Lemma~\ref{lem:sphere-basic},
\begin{align*}
 \E A(\Theta)^2
 &=n\E\Theta_1^6\\
 &=n\frac{5!!}{n(n+2)(n+4)}\\
 &=\frac{15}{(n+2)(n+4)}.
\end{align*}
This proves \eqref{eq:A2}.

Next expand
\[
 A(\Theta)^4
 =\sum_{i,j,k,l=1}^n
   \Theta_i^3\Theta_j^3\Theta_k^3\Theta_l^3.
\]
By \eqref{eq:odd-coordinate-zero}, a summand can have nonzero expectation
only if every coordinate index among $i,j,k,l$ occurs an even number of
times.  There are exactly two possible patterns:
\begin{enumerate}[label=\textnormal{(\roman*)},leftmargin=2.2em]
 \item one index occurs four times;
 \item two distinct indices occur twice each.
\end{enumerate}
The first pattern gives $n$ terms.  For the second pattern, choose the two
indices in $\binom n2$ ways and choose the two positions occupied by the
first index in $\binom42=6$ ways.  Therefore
\begin{align*}
 \E A(\Theta)^4
 &=n\E\Theta_1^{12}
   +6\binom n2\E(\Theta_1^6\Theta_2^6).
\end{align*}
Lemma~\ref{lem:sphere-basic} gives
\begin{align*}
 \E\Theta_1^{12}
 &=\frac{11!!}{n(n+2)(n+4)(n+6)(n+8)(n+10)}\\
 &=\frac{10395}{n(n+2)(n+4)(n+6)(n+8)(n+10)},
\end{align*}
and
\begin{align*}
 \E(\Theta_1^6\Theta_2^6)
 &=\frac{(5!!)^2}{n(n+2)(n+4)(n+6)(n+8)(n+10)}\\
 &=\frac{225}{n(n+2)(n+4)(n+6)(n+8)(n+10)}.
\end{align*}
Consequently,
\begin{align*}
 \E A(\Theta)^4
 &=\frac{10395n+675n(n-1)}
 {n(n+2)(n+4)(n+6)(n+8)(n+10)}\\
 &=\frac{135(5n+72)}
 {(n+2)(n+4)(n+6)(n+8)(n+10)},
\end{align*}
which is \eqref{eq:A4}.

Finally, exchangeability and Lemma~\ref{lem:sphere-basic} yield
\begin{align*}
 \E B(\Theta)
 &=\sum_{j=1}^n\E\Theta_j^4
 =n\E\Theta_1^4\\
 &=n\frac{3}{n(n+2)}
 =\frac{3}{n+2}.
\end{align*}
This proves \eqref{eq:Bmean}.

For completeness, interpolation between $L^1$ and $L^4$ gives
\[
 \|A\|_2\le\|A\|_1^{1/3}\|A\|_4^{2/3},
\]
because $1/2=(1/3)/1+(2/3)/4$.  Cubing this inequality and rearranging gives
\[
 \E|A|
 =\|A\|_1
 \ge\frac{\|A\|_2^3}{\|A\|_4^2}
 =\frac{(\E A^2)^{3/2}}{(\E A^4)^{1/2}}.
\]
Also, monotonicity of $L^p$ norms gives
\[
 \E|A|^3=\|A\|_3^3\le\|A\|_4^3=(\E A^4)^{3/4}.
\]
Substituting \eqref{eq:A2} and \eqref{eq:A4} into these two inequalities
proves \eqref{eq:A1-lower}.
\end{proof}

\begin{proof}[Proof of Theorem~\ref{thm:lower}]
Average \eqref{eq:cf-lower} over $\Theta$. By
Lemma~\ref{lem:sphere-lower},
\[
 \E W_1(\Law(S_\Theta),\gamma_1)
 \ge \frac{c_1t^2}{n}-\frac{C_1t^3}{n}-\frac{C_2t^8}{n^3}.
\]
Choose $t>0$ so small that $C_1t\le c_1/2$. This choice is independent of
$n$. Then
\[
 \E W_1(\Law(S_\Theta),\gamma_1)
 \ge \frac{c_1t^2}{2n}-\frac{C_2t^8}{n^3}.
\]
For all sufficiently large $n$, the last term is at most
$c_1t^2/(4n)$. The desired lower bound follows.
\end{proof}

\section{External inputs and consequences used below}

Throughout the proof, $C,c>0$ denote universal constants which may change
from line to line. Constants with a subscript may depend only on that subscript.

\subsection{Notation for the upper-bound proof}

From this section onward, $X$ is the centered isotropic log-concave random
vector fixed in the introduction, and $Y$ is an independent copy of $X$.
In particular, both $X$ and $Y$ have the same log-concave law. We use the
persistent notation
\begin{equation}\label{eq:upper-notation}
\begin{aligned}
 R_X&:=|X|, & R_Y&:=|Y|,\\
 \Delta_X&:=R_X^2-n, & \Delta_Y&:=R_Y^2-n,\\
 Z&:=\ip{X}{Y}, & \widetilde X&:=\frac{X-Y}{\sqrt2}.
\end{aligned}
\end{equation}
The symbols in \eqref{eq:upper-notation} retain these meanings throughout
the upper-bound proof.

\subsection{Quadratic variance}

The following estimate is an immediate isotropic consequence of
\cite[Theorem~1.2]{Letwin}.

\begin{theorem}\label{thm:qve}
Let $X$ be centered, isotropic, and log-concave in $\R^n$. Then, for every
symmetric matrix $M$,
\begin{equation}\label{eq:qve}
 \Var(X^TMX)\le 8\Tr(M^2).
\end{equation}
\end{theorem}

Indeed, the cited theorem states
\[
 \Var(X^TMX)\le2\E|\nabla(X^TMX)|^2.
\]
For symmetric $M$, $\nabla(X^TMX)=2MX$. Moreover, isotropy gives
\begin{align*}
 \E|MX|^2
 &=\E\Tr(MXX^TM)\\
 &=\Tr\bigl(M\E[XX^T]M\bigr)\\
 &=\Tr(M^2).
\end{align*}
Substitution proves \eqref{eq:qve}. We now derive the three consequences
needed later.

Taking $M=\Id$ in \eqref{eq:qve} gives
\begin{equation}\label{eq:delta-L2}
 \E\Delta_X^2\le 8n.
\end{equation}

For $u\in\R^n$, define the symmetric matrix
\begin{equation}\label{eq:Tu-def}
 T_u:=\E\bigl[\ip{X}{u}XX^T\bigr].
\end{equation}

\begin{lemma}\label{lem:Tu}
For every $u\in\R^n$,
\begin{equation}\label{eq:Tu}
 \norm{T_u}_{\HS}\le\sqrt8\,|u|.
\end{equation}
Consequently, if
\begin{equation}\label{eq:v-def}
 v:=\E[\Delta_X X],
\end{equation}
then
\begin{equation}\label{eq:v-bound}
 |v|^2\le8n.
\end{equation}
\end{lemma}

\begin{proof}
Let $M$ be symmetric with $\norm M_{\HS}=1$. Since $\E\ip{X}{u}=0$,
\[
 \Tr(T_uM)
 =\E\left[\ip{X}{u}\bigl(X^TMX-\Tr M\bigr)\right].
\]
Cauchy--Schwarz and Theorem~\ref{thm:qve} give
\[
 |\Tr(T_uM)|
 \le
 \sqrt{\E\ip{X}{u}^2}\,
 \sqrt{\Var(X^TMX)}
 \le \sqrt8\,|u|.
\]
For a symmetric matrix $A$,
\[
 \|A\|_{\HS}
 =\sup\bigl\{|\Tr(AM)|:M=M^T,\ \|M\|_{\HS}=1\bigr\}.
\]
Apply this formula with $A=T_u$. The preceding estimate therefore proves
\eqref{eq:Tu}.
Moreover,
\[
 \ip{v}{u}
 =\E[\Delta_X\ip{X}{u}]
 =\E[R_X^2\ip{X}{u}]
 =\Tr T_u.
\]
Hence
\[
 |\ip{v}{u}|
 \le\sqrt n\norm{T_u}_{\HS}
 \le\sqrt{8n}|u|,
\]
Taking the supremum over all $u$ with $|u|=1$ gives
$|v|\le\sqrt{8n}$, and hence \eqref{eq:v-bound}.
\end{proof}

Define also
\begin{equation}\label{eq:B-def}
 B:=\E[\Delta_X XX^T].
\end{equation}

\begin{lemma}\label{lem:B}
The matrix $B$ satisfies
\begin{equation}\label{eq:B-bound}
 \norm B_{\HS}\le8\sqrt n.
\end{equation}
\end{lemma}

\begin{proof}
For every symmetric $M$ with $\norm M_{\HS}=1$,
\[
 \Tr(BM)
 =\Cov(R_X^2,X^TMX).
\]
Thus, by Cauchy--Schwarz, \eqref{eq:delta-L2}, and
Theorem~\ref{thm:qve},
\[
 |\Tr(BM)|
 \le\sqrt{8n}\sqrt8=8\sqrt n.
\]
Using
\[
 \|B\|_{\HS}
 =\sup\bigl\{|\Tr(BM)|:M=M^T,\ \|M\|_{\HS}=1\bigr\}
\]
proves \eqref{eq:B-bound}.
\end{proof}

\subsection{Polynomial moments}

We use the following fixed-degree consequence of a polynomial norm
comparison. The derivation records the dependence on $p$ explicitly.

\begin{theorem}\label{thm:CW}
Let $\nu$ be a log-concave probability measure on a finite-dimensional
Euclidean space. For every $p\ge2$ and every polynomial $P$ of degree at
most two,
\begin{equation}\label{eq:CW}
 \norm{P}_{L^p(\nu)}\le Cp^2\norm{P}_{L^2(\nu)},
\end{equation}
where $C$ is universal.
\end{theorem}

\begin{proof}
By \cite[Theorem~7]{CarberyWright}, applied with $d=2$, $q=2p$, and $r=4$,
we have
\[
 (\E|P|^p)^{1/(2p)}
 \le Cp(\E|P|^2)^{1/4}.
\]
Squaring both sides yields
\[
 (\E|P|^p)^{1/p}\le Cp^2(\E|P|^2)^{1/2},
\]
which is \eqref{eq:CW}.
\end{proof}

Only finitely many fixed exponents are used below. Since the laws of $X$
and $Y$ have densities $e^{-\mathcal V_X}$ and $e^{-\mathcal V_Y}$ with
convex potentials, their product density on $\R^{2n}$ is
\[
 \exp\bigl(-\mathcal V_X(x)-\mathcal V_Y(y)\bigr),
\]
and is log-concave. The variable
$Z=\ip{X}{Y}$ is a polynomial of degree two. Isotropy and independence give
\[
 \E Z^2=n.
\]
Together with \eqref{eq:delta-L2} and Theorem~\ref{thm:CW}, this yields,
for every fixed $p\ge2$,
\begin{equation}\label{eq:moment-ledger}
 \norm{\Delta_X}_p\le C_p\sqrt n,
 \qquad
 \norm Z_p\le C_p\sqrt n.
\end{equation}
We also need moments of $R_X$.  For $p\ge4$,
\begin{align*}
 \|R_X\|_p^2
 &=\|R_X^2\|_{p/2}
 =\|n+\Delta_X\|_{p/2}\\
 &\le n+\|\Delta_X\|_{p/2}
 \le C_p n.
\end{align*}
Thus $\|R_X\|_p\le C_p\sqrt n$ for $p\ge4$.  If $2\le p<4$, then
$\|R_X\|_p\le\|R_X\|_4$.  Hence
\begin{equation}\label{eq:Rmom}
 \norm{R_X}_p\le C_p\sqrt n
 \qquad(p\ge2\text{ fixed}).
\end{equation}

\subsection{Radial deviations}

The next estimate follows from
\cite[Theorem~1.1 and the paragraph following it]{GuedonMilman}: take
$A=I_n$ and $\alpha=1$, and use the universal $\psi_1$ bound for
log-concave random vectors recorded there.

\begin{theorem}\label{thm:GM}
If $W$ is centered, isotropic, and log-concave in $\R^n$, then
\begin{equation}\label{eq:GM}
 \Pp\left\{\bigl||W|-\sqrt n\bigr|\ge t\sqrt n\right\}
 \le C\exp\bigl[-c\sqrt n\min(t^3,t)\bigr]
 \qquad(t\ge0).
\end{equation}
\end{theorem}

We shall use \eqref{eq:GM} only for fixed positive values of $t$.

\subsection{The difference of two independent copies}

The variables $X$ and $Y$ were fixed in \eqref{eq:upper-notation}; each has
the same log-concave law. Since reflection is an invertible linear map, $-Y$
is log-concave. By \cite[Proposition~3.5]{SaumardWellner}, the convolution
$X-Y=X+(-Y)$ is therefore log-concave. Independence and centering give
\[
 \E\widetilde X=0,
\]
and
\begin{align*}
 \E\widetilde X\widetilde X^T
 &=\frac12\bigl(\E XX^T+\E YY^T-\E X\,\E Y^T-\E Y\,\E X^T\bigr)\\
 &=I_n.
\end{align*}
Thus $\widetilde X$ is centered, isotropic, and log-concave, so
Theorem~\ref{thm:GM} applies to it.

\section{The law after averaging the direction}

Let $\bar\mu$ be the averaged law, defined by
\begin{equation}\label{eq:barmu}
 \bar\mu(E)=\int_{S^{n-1}}\mu_\theta(E)\,\dd\sigma_{n-1}(\theta)
\end{equation}
for every Borel set $E\subset\R$. This law depends only on the radius
$R_X=|X|$ fixed in \eqref{eq:upper-notation}.

Let $U$ be uniform on $S^{n-1}$ and define
\begin{equation}\label{eq:eta-n}
 \eta_n:=\sqrt n\,U_1.
\end{equation}
Choose $\eta_n$ independent of $R_X$.  For a fixed nonzero vector $x$, choose
an orthogonal matrix $O$ such that $Oe_1=x/|x|$. Since $O^T\Theta$ is
uniform on $S^{n-1}$,
\[
 \ip{x}{\Theta}
 =|x|\ip{e_1}{O^T\Theta}
 \stackrel{d}{=}|x|U_1.
\]
The conditional law of $\ip{X}{\Theta}$ given $X=x$ therefore depends on
$x$ only through $|x|$.  Integrating this conditional law in $X$ gives
\begin{equation}\label{eq:averaged-rep}
 \bar\mu
 =\Law\left(\frac{R_X}{\sqrt n}\eta_n\right).
\end{equation}

We now prove the required $n^{-1}$ estimate.

\begin{lemma}\label{lem:sphere-gauss}
For all $n\ge6$,
\begin{equation}\label{eq:sphere-gauss}
 W_1(\Law(\eta_n),\gamma_1)\le\frac Cn.
\end{equation}
\end{lemma}

\begin{proof}
By \eqref{eq:U1-density} in Lemma~\ref{lem:sphere-basic}, $U_1$ has density
$c_n(1-u^2)^{(n-3)/2}$ on $[-1,1]$. Since $\eta_n=\sqrt n\,U_1$, the change of
variables $t=\sqrt n\,u$ gives
\begin{equation}\label{eq:fn}
 f_n(t)=\kappa_n\left(1-\frac{t^2}{n}\right)_+^{(n-3)/2},
 \qquad
 \kappa_n=
 \frac{\Gamma(n/2)}{\sqrt{\pi n}\,\Gamma((n-1)/2)}.
\end{equation}
The general gamma-ratio expansion in
\cite[Eq.~(5.11.13)]{NIST}, specialized to the present parameters, gives
\[
 \frac{\Gamma(x+1/2)}{\Gamma(x)}
 =x^{1/2}\left(1-\frac{1}{8x}+O(x^{-2})\right)
 \qquad(x\to\infty).
\]
Apply it with $x=(n-1)/2$.  Substitution into \eqref{eq:fn} gives
\begin{equation}\label{eq:kappa}
 \kappa_n=(2\pi)^{-1/2}\bigl(1+O(n^{-1})\bigr).
\end{equation}
Let $\phi(t)=(2\pi)^{-1/2}e^{-t^2/2}$.

For $|t|\le n^{1/8}$, put $u=t^2/n$.  From \eqref{eq:fn},
\[
 \log\frac{f_n(t)}{\phi(t)}
 =\log(\sqrt{2\pi}\,\kappa_n)
  +\frac{n-3}{2}\log(1-u)+\frac{t^2}{2}.
\]
Here $0\le u\le n^{-3/4}$.  The expansion
$\log(1-u)=-u-u^2/2+O(u^3)$ and \eqref{eq:kappa} therefore give
\[
 \left|
 \log\frac{f_n(t)}{\phi(t)}
 \right|
 \le \frac Cn(1+t^2+t^4).
\]
In this range the right-hand side is $O(n^{-1/2})$. Hence
$|e^u-1|\le e^{|u|}|u|$, applied to
$u=\log(f_n/\phi)$, gives
\begin{equation}\label{eq:local-density}
 |f_n(t)-\phi(t)|
 \le \frac Cn(1+t^4)e^{-ct^2}
 \qquad(|t|\le n^{1/8}).
\end{equation}
For $|t|\le\sqrt n$, the inequality $\log(1-u)\le-u$ gives
\[
 \left(1-\frac{t^2}{n}\right)^{(n-3)/2}
 \le \exp\left(-\frac{n-3}{2n}t^2\right)
 \le e^{-t^2/4}
 \qquad(n\ge6).
\]
Together with the uniform bound $\kappa_n\le C$ from \eqref{eq:kappa}, this
shows $f_n(t)\le Ce^{-t^2/4}$. Hence
\begin{equation}\label{eq:density-tail}
 \int_{|t|>n^{1/8}}|t|\bigl(f_n(t)+\phi(t)\bigr)\dd t
 \le Ce^{-c n^{1/4}}.
\end{equation}
Combining \eqref{eq:local-density} and \eqref{eq:density-tail} yields
\begin{equation}\label{eq:weighted-density}
 \int_\R |t|\,|f_n(t)-\phi(t)|\dd t\le\frac Cn.
\end{equation}

Let $h=f_n-\phi$. Since both $f_n$ and $\phi$ integrate to one,
$\int_\R h=0$. Put $D_n(t)=\int_{-\infty}^t h(s)\,\dd s$. For $t\ge0$ we may
also write $D_n(t)=-\int_t^\infty h(s)\,\dd s$. Therefore
\begin{align*}
 \int_\R|D_n(t)|\,\dd t
 &\le\int_{-\infty}^0\int_{-\infty}^t|h(s)|\,\dd s\,\dd t
 +\int_0^\infty\int_t^\infty|h(s)|\,\dd s\,\dd t\\
 &=\int_\R|s|\,|h(s)|\,\dd s,
\end{align*}
where Tonelli's theorem was used in the last equality. By
Lemma~\ref{lem:w1-line}, the left-hand side is
$W_1(\Law(\eta_n),\gamma_1)$. Now use \eqref{eq:weighted-density}.
\end{proof}

\begin{lemma}\label{lem:scale}
Let
\[
 \tau:=\frac{R_X}{\sqrt n}.
\]
Then
\begin{equation}\label{eq:scale-result}
 W_1(\Law(\tau\eta_n),\Law(\eta_n))\le\frac Cn.
\end{equation}
\end{lemma}

\begin{proof}
Since $\E\tau^2=1$,
\begin{equation}\label{eq:mean-scale}
 \E(\tau-1)=-\frac12\E(\tau-1)^2.
\end{equation}
Moreover,
\begin{equation}\label{eq:scale-L2}
 \E(\tau-1)^2
 =\frac1n\E(R_X-\sqrt n)^2
 \le\frac1{n^2}\E(R_X^2-n)^2
 \le\frac8n.
\end{equation}

Let $F_n$ be the distribution function of $\eta_n$ and, for $s>0$, set
\[
 F_{n,s}(t):=F_n(t/s).
\]
Thus $F_{n,s}$ is the distribution function of $s\eta_n$. For $n\ge6$,
differentiation of the explicit density in \eqref{eq:fn} gives
\begin{align}
 \partial_s F_{n,s}(t)
 &=-\frac{t}{s^2}f_n(t/s),\label{eq:H1}\\
 \partial_s^2F_{n,s}(t)
 &=\frac{2t}{s^3}f_n(t/s)
 +\frac{t^2}{s^4}f_n'(t/s).\label{eq:H2}
\end{align}
For $1/2\le s\le3/2$,
\begin{equation}\label{eq:H-derivative-L1}
 \int_\R|\partial_sF_{n,s}(t)|\dd t\le C,
 \qquad
 \int_\R|\partial_s^2F_{n,s}(t)|\dd t\le C.
\end{equation}
Indeed, after the change of variables $t=su$,
\[
 \int_\R|\partial_sF_{n,s}(t)|\,\dd t
 =\int_\R|u|f_n(u)\,\dd u=\E|\eta_n|\le1,
\]
where the last inequality follows from $\E\eta_n^2=1$. For the second
derivative, the triangle inequality and the same change of variables reduce
the estimate to $\E|\eta_n|$ and
\[
 \int_\R u^2|f_n'(u)|\dd u
 =2\int_0^{\sqrt n}u^2(-f_n'(u))\dd u
 =4\int_0^{\sqrt n}u f_n(u)\dd u
 =2\E|\eta_n|.
\]
The integration by parts has no boundary term because $f_n(\sqrt n)=0$ and
$u^2f_n(u)=0$ at $u=0$.

Let
\[
 A_\tau:=\{1/2\le\tau\le3/2\}.
\]
For every fixed $s\in[1/2,3/2]$, Taylor's formula with integral remainder,
applied to the map $r\mapsto F_{n,r}(t)$ at $r=1$, gives
\[
 F_{n,s}(t)-F_{n,1}(t)
 =(s-1)\left.\partial_rF_{n,r}(t)\right|_{r=1}
 +(s-1)^2\int_0^1(1-u)
 \partial_r^2F_{n,1+u(s-1)}(t)\,\dd u.
\]
Substitute $s=\tau$, multiply by $\mathbf1_{A_\tau}$, take expectation, and
then integrate in $t$. By \eqref{eq:H-derivative-L1},
\begin{align}
 &\int_\R\left|
 \E\bigl[(F_{n,\tau}(t)-F_{n,1}(t))\mathbf1_{A_\tau}\bigr]
 \right|\dd t\notag\\
 &\hspace{2cm}\le
 C\left|\E[(\tau-1)\mathbf1_{A_\tau}]\right|
 +C\E(\tau-1)^2.\label{eq:scale-good}
\end{align}
Using \eqref{eq:mean-scale},
\begin{align*}
 \left|\E[(\tau-1)\mathbf1_{A_\tau}]\right|
 &\le \frac12\E(\tau-1)^2
   +\E[|\tau-1|\mathbf1_{A_\tau^c}]\\
 &\le \frac4n
   +\|\tau-1\|_2\,\Pp(A_\tau^c)^{1/2}.
\end{align*}
The event $A_\tau^c$ is
\[
 \{|R_X-\sqrt n|>\tfrac12\sqrt n\}.
\]
Theorem~\ref{thm:GM}, applied to $X$ with $t=1/2$, therefore gives
$\Pp(A_\tau^c)\le Ce^{-c\sqrt n}$. Together with \eqref{eq:scale-L2}, this
proves
\begin{equation}\label{eq:truncated-scale-mean}
 \left|\E[(\tau-1)\mathbf1_{A_\tau}]\right|
 \le\frac Cn+Ce^{-c\sqrt n}.
\end{equation}

For fixed $s>0$, the generalized inverse of the distribution function of
$s\eta_n$ is $sF_n^{-1}$. The monotone-coupling formula in
Lemma~\ref{lem:w1-line} gives
\begin{align*}
 W_1(\Law(s\eta_n),\Law(\eta_n))
 &=\int_0^1|sF_n^{-1}(r)-F_n^{-1}(r)|\,\dd r\\
 &=|s-1|\int_0^1|F_n^{-1}(r)|\,\dd r\\
 &=|s-1|\E|\eta_n|.
\end{align*}
Using also the distribution-function formula in Lemma~\ref{lem:w1-line},
\begin{equation}\label{eq:scaled-cdf-distance}
 \int_\R|F_{n,s}(t)-F_{n,1}(t)|\dd t
 =|s-1|\E|\eta_n|.
\end{equation}
Apply \eqref{eq:scaled-cdf-distance} with $s=\tau$ on $A_\tau^c$. Then
Cauchy--Schwarz, \eqref{eq:scale-L2}, and the bound on
$\Pp(A_\tau^c)$ give
\begin{align}
 &\E\left[\mathbf1_{A_\tau^c}
 \int_\R|F_{n,\tau}(t)-F_{n,1}(t)|\dd t\right]\notag\\
 &=\E|\eta_n|\,\E[|\tau-1|\mathbf1_{A_\tau^c}]\notag\\
 &\le C\|\tau-1\|_2\Pp(A_\tau^c)^{1/2}
 \le Ce^{-c\sqrt n}.\label{eq:scale-bad}
\end{align}
Combining \eqref{eq:scale-good}, \eqref{eq:truncated-scale-mean},
\eqref{eq:scale-L2}, and \eqref{eq:scale-bad} proves
\eqref{eq:scale-result}.
\end{proof}

\begin{proposition}\label{prop:averaged}
The averaged law satisfies
\begin{equation}\label{eq:averaged}
 W_1(\bar\mu,\gamma_1)\le\frac Cn.
\end{equation}
\end{proposition}

\begin{proof}
By \eqref{eq:averaged-rep}, $\bar\mu=\Law(\tau\eta_n)$.  Hence
\[
 W_1(\bar\mu,\gamma_1)
 \le W_1(\Law(\tau\eta_n),\Law(\eta_n))
   +W_1(\Law(\eta_n),\gamma_1).
\]
The two terms are bounded by Lemmas~\ref{lem:scale} and
\ref{lem:sphere-gauss}, respectively.
\end{proof}

\section{An exact kernel for directional fluctuations}\label{sec:kernel}

\subsection{A weighted distribution-function integral}

For probability measures $\mu,\nu$ on $\R$ with finite third moments, define
\begin{equation}\label{eq:J-def}
 \cJ(\mu,\nu)
 :=\int_\R(1+t^2)\bigl(F_\mu(t)-F_\nu(t)\bigr)^2\dd t.
\end{equation}

\begin{lemma}\label{lem:CDF-reduction}
For all such $\mu,\nu$,
\begin{equation}\label{eq:W1-J}
 W_1(\mu,\nu)\le\sqrt\pi\,\cJ(\mu,\nu)^{1/2}.
\end{equation}
\end{lemma}

\begin{proof}
Lemma~\ref{lem:w1-line} gives
\[
 W_1(\mu,\nu)=\int_\R|F_\mu(t)-F_\nu(t)|\,\dd t.
\]
Write the integrand as
\[
 \frac{1}{\sqrt{1+t^2}}
 \bigl(\sqrt{1+t^2}\,|F_\mu(t)-F_\nu(t)|\bigr)
\]
and apply Cauchy--Schwarz. Since
$\int_\R(1+t^2)^{-1}\,\dd t=\pi$, this gives
\eqref{eq:W1-J}.
\end{proof}

Let
\begin{equation}\label{eq:Psi}
 \Psi(t):=t+\frac{t^3}{3},
 \qquad
 K(u,v):=|\Psi(u)-\Psi(v)|.
\end{equation}

\begin{lemma}\label{lem:cdf-identity}
Let $U,U'$ be independent with law $\mu$, and let $V,V'$ be independent
with law $\nu$, all four variables independent. Then
\begin{equation}\label{eq:cdf-identity}
 \cJ(\mu,\nu)
 =\E K(U,V)
 -\frac12\E K(U,U')
 -\frac12\E K(V,V').
\end{equation}
\end{lemma}

\begin{proof}
The map $\Psi$ is strictly increasing and $\Psi'(t)=1+t^2$.  Therefore
\[
 F_{\Psi(U)}(\Psi(t))=\Pp\{\Psi(U)\le\Psi(t)\}=F_U(t),
\]
and the same identity holds for $V$.  With $s=\Psi(t)$ and
$\dd s=(1+t^2)\dd t$, we obtain
\[
 \cJ(\mu,\nu)
 =\int_\R\bigl(F_{\Psi(U)}(s)-F_{\Psi(V)}(s)\bigr)^2\dd s.
\]
We prove the needed identity. For real numbers $a,b$,
\[
 |a-b|=\int_\R
 \bigl(\mathbf1_{\{a\le s\}}-\mathbf1_{\{b\le s\}}\bigr)^2\,\dd s.
\]
Let $P,P'$ be independent copies and let $Q,Q'$ be independent copies, with
all four variables independent. Tonelli's theorem gives
\begin{align*}
 \E|P-Q|
 &=\int_\R\E
 \bigl(\mathbf1_{\{P\le s\}}-\mathbf1_{\{Q\le s\}}\bigr)^2\,\dd s\\
 &=\int_\R\bigl(F_P(s)+F_Q(s)-2F_P(s)F_Q(s)\bigr)\,\dd s.
\end{align*}
Similarly,
\[
 \frac12\E|P-P'|
 =\int_\R F_P(s)(1-F_P(s))\,\dd s
\]
and
\[
 \frac12\E|Q-Q'|
 =\int_\R F_Q(s)(1-F_Q(s))\,\dd s.
\]
Subtracting the last two formulas from the first gives
\[
 \int_\R(F_P(s)-F_Q(s))^2\,\dd s.
\]
The finite-third-moment assumption implies
$\E|\Psi(U)|+\E|\Psi(V)|<\infty$, so every displayed integral is finite.
Apply the identity with $P=\Psi(U)$ and $Q=\Psi(V)$.
\end{proof}

Define the averaged weighted squared error
\begin{equation}\label{eq:En-def}
 \cE_n:=\E_\Theta\cJ(\mu_\Theta,\bar\mu).
\end{equation}

\begin{lemma}
\label{lem:same-independent}
Let $\Theta,\Theta'$ be independent uniform directions, independent of
$(X,Y)$. Then
\begin{align}
 \cE_n
 =\frac12\bigl(&\E K(\ip{X}{\Theta},\ip{Y}{\Theta'})
 \notag\\
 &-\E K(\ip{X}{\Theta},\ip{Y}{\Theta})\bigr).
 \label{eq:same-independent}
\end{align}
\end{lemma}

\begin{proof}
Let $\Theta,\Theta',\Theta''$ be independent uniform directions, also
independent of $(X,Y)$. Lemma
\ref{lem:cdf-identity}, followed by averaging in $\Theta$, gives
\begin{align*}
 \cE_n
 &=\E K(\ip{X}{\Theta},\ip{Y}{\Theta'})
 -\frac12\E K(\ip{X}{\Theta},\ip{Y}{\Theta})\\
 &\quad-\frac12\E K(\ip{X}{\Theta'},\ip{Y}{\Theta''}).
\end{align*}
The first and third expectations are equal because both use two independent
uniform directions. Combining their coefficients gives
\eqref{eq:same-independent}.
\end{proof}

We now compute the same-direction formula exactly. For fixed $x,y\in\R^n$,
write
\begin{equation}\label{eq:invariants}
 a=|x|^2,
 \qquad b=|y|^2,
 \qquad z=\ip{x}{y},
 \qquad A=a+b,
 \qquad D=a-b,
\end{equation}
and let $m_n=\E|\Theta_1|$.  Integrating the density in
\eqref{eq:U1-density} gives
\begin{equation}\label{eq:mn}
 m_n=\frac{\Gamma(n/2)}{\sqrt\pi\,\Gamma((n+1)/2)}.
\end{equation}
The upper bound $m_n\le n^{-1/2}$ follows from Cauchy--Schwarz and
$\E\Theta_1^2=1/n$.  For the lower bound, H\"older's inequality gives
\[
 \E\Theta_1^2
 \le (\E|\Theta_1|)^{2/3}(\E\Theta_1^4)^{1/3}.
\]
Since Lemma~\ref{lem:sphere-basic} gives
$\E\Theta_1^2=1/n$ and $\E\Theta_1^4=3/[n(n+2)]$, rearranging yields
\[
 m_n\ge \frac{\sqrt{n+2}}{\sqrt3\,n}.
\]
Consequently
\begin{equation}\label{eq:mn-bounds}
 c n^{-1/2}\le m_n\le n^{-1/2}.
\end{equation}

\begin{proposition}\label{prop:formula}
For $A-2z=|x-y|^2>0$,
\begin{align}
 S_n(a,b,z)
 &:=\E_\Theta
 K(\ip{x}{\Theta},\ip{y}{\Theta})\notag\\
 &=m_n\left[
 \sqrt{A-2z}
 +\frac{(5A+2z)\sqrt{A-2z}}{12(n+1)}
 +\frac{D^2}{4(n+1)\sqrt{A-2z}}
 \right].\label{eq:formula}
\end{align}
The right-hand side extends continuously to $x=y$, where it is zero.
\end{proposition}

\begin{proof}
Put
\[
 p=\ip{x}{\Theta},\qquad q=\ip{y}{\Theta},\qquad d=x-y,
 \qquad r=|d|=\sqrt{A-2z}.
\]
Since
\[
 \Psi(p)-\Psi(q)
 =(p-q)\left(1+\frac{p^2+pq+q^2}{3}\right)
\]
and $p^2+pq+q^2\ge0$, we have
\begin{equation}\label{eq:K-factor}
 K(p,q)
 =|\ip{d}{\Theta}|
 \left(1+\frac{p^2+pq+q^2}{3}\right).
\end{equation}

Assume first that $d\ne0$ and put $e=d/r$. We compute the matrix
\[
 M_e:=\E\bigl[|\ip{e}{\Theta}|\Theta\Theta^T\bigr].
\]
Choose coordinates so that $e=e_1$. Sign changes of the coordinates show
that all off-diagonal entries are zero. Rotations fixing $e_1$ show that the
last $n-1$ diagonal entries are equal. If their common value is
$\lambda_\perp$ and the first entry is $\lambda_\parallel$, then
$\lambda_\parallel=\E|\Theta_1|^3$.  From the density
\eqref{eq:U1-density}, the substitution $v=u^2$ gives
\begin{equation}\label{eq:absolute-sphere-moment}
 \E|\Theta_1|^r
 =\frac{\Gamma(n/2)\Gamma((r+1)/2)}
 {\sqrt\pi\,\Gamma((n+r)/2)}
 \qquad(r>-1).
\end{equation}
Taking $r=1$ and $r=3$ in this formula gives
\[
 \lambda_\parallel=\frac{2m_n}{n+1}.
\]
Taking the trace gives
\[
 \lambda_\parallel+(n-1)\lambda_\perp
 =\E|\Theta_1|=m_n,
\]
so $\lambda_\perp=m_n/(n+1)$. Therefore
\begin{equation}\label{eq:sphere-matrix}
 M_e=\frac{m_n}{n+1}(I_n+ee^T).
\end{equation}

Let
\[
 M=xx^T+\frac{xy^T+yx^T}{2}+yy^T.
\]
Then $p^2+pq+q^2=\Theta^TM\Theta$. Formula
\eqref{eq:sphere-matrix} gives
\begin{align*}
 \E\bigl[|\ip{d}{\Theta}|(p^2+pq+q^2)\bigr]
 &=r\Tr(MM_e)\\
 &=\frac{rm_n}{n+1}\bigl(\Tr M+e^TMe\bigr).
\end{align*}
We now calculate the two scalar terms. First,
\[
 \Tr M=a+b+z=A+z.
\]
Second,
\[
 e^TMe
 =\frac{(a-z)^2+(a-z)(z-b)+(z-b)^2}{r^2}.
\]
Since
\[
 a-z=\frac{D+r^2}{2},
 \qquad
 z-b=\frac{D-r^2}{2},
\]
the numerator equals $(3D^2+r^4)/4$. Hence
\[
 \Tr M+e^TMe
 =A+z+\frac{r^2}{4}+\frac{3D^2}{4r^2}
 =\frac{5A+2z}{4}+\frac{3D^2}{4(A-2z)}.
\]
Insert this identity into \eqref{eq:K-factor}. Since
$\E|\ip{d}{\Theta}|=rm_n$, we obtain \eqref{eq:formula}.

It remains to discuss $d=0$. In this case $x=y$, so $D=0$ and
$K(\ip{x}{\Theta},\ip{y}{\Theta})=0$. More generally,
\[
 |D|=|\ip{x-y}{x+y}|\le |x-y|\,|x+y|.
\]
Thus $D^2/|x-y|\le |x-y|\,|x+y|^2$, which tends to zero as $y\to x$.
Therefore the formula extends continuously to $x=y$ with value zero.
\end{proof}

\subsection{Independent spherical directions}

Let $U,V$ be independent and uniform on $S^{n-1}$ and independent of
$(X,Y)$. Using the radii fixed in \eqref{eq:upper-notation}, define
\begin{equation}\label{eq:circ}
 X^\circ=R_XU,
 \qquad
 Y^\circ=R_YV,
 \qquad
 Z^\circ=\ip{X^\circ}{Y^\circ}.
\end{equation}
Set
\[
 a=R_X^2,
 \qquad b=R_Y^2,
 \qquad A=a+b,
 \qquad D=a-b.
\]
These definitions agree with the invariants in \eqref{eq:invariants}.

\begin{lemma}\label{lem:cdf-comparison}
One has
\begin{equation}\label{eq:cdf-comparison}
 \cE_n
 =\frac12\left(
 \E S_n(a,b,Z^\circ)-\E S_n(a,b,Z)
 \right).
\end{equation}
\end{lemma}

\begin{proof}
Fix the radii $R_X$ and $R_Y$. In the first expectation in
\eqref{eq:same-independent}, the variables
$\ip{X}{\Theta}$ and $\ip{Y}{\Theta'}$ are independent and have laws
$R_XU_1$ and $R_YV_1$, where $U_1,V_1$ are independent first coordinates of
uniform sphere points. In the pair
\[
 (\ip{X^\circ}{\Theta},\ip{Y^\circ}{\Theta}),
\]
condition on the common direction $\Theta$. The vectors $U$ and $V$ are
independent and uniform, so the two inner products are again independent
with laws $R_XU_1$ and $R_YV_1$. Thus the two pairs have the same conditional
law given $(R_X,R_Y)$.

The second expectation in \eqref{eq:same-independent} uses the original
vectors $X,Y$ and one common direction. Applying Proposition
\ref{prop:formula} to both expectations gives \eqref{eq:cdf-comparison}.
\end{proof}

Conditionally on $a,b$,
\[
 Z^\circ=\sqrt{ab}\,\ip{U}{V}.
\]
Fix $V=w$. Choose an orthogonal matrix $O$ with $Oe_1=w$. Since $O^TU$ is
uniform on the sphere,
\[
 \ip{U}{w}=\ip{O^TU}{e_1}
\]
has the same law as $U_1$.  Lemma~\ref{lem:sphere-basic} now gives
\begin{equation}\label{eq:circ-moments}
 \E[Z^\circ\mid a,b]=0,
 \qquad
 \E[(Z^\circ)^2\mid a,b]=\frac{ab}{n},
 \qquad
 \E[(Z^\circ)^3\mid a,b]=0,
\end{equation}
and
\begin{equation}\label{eq:circ-fourth}
 \E[(Z^\circ)^4\mid a,b]
 =\frac{3a^2b^2}{n(n+2)}.
\end{equation}

For the original inner product, we use the following cubic moment identity.

\begin{lemma}\label{lem:Z3}
One has
\begin{equation}\label{eq:Z3}
 \E Z^3
 =\norm{\E X^{\otimes3}}_{\HS}^2
 =\sum_{k=1}^n\norm{T_{e_k}}_{\HS}^2
 \le8n.
\end{equation}
\end{lemma}

\begin{proof}
Expanding $Z^3=\ip{X}{Y}^3$ and using independence gives
\[
 \E Z^3
 =\sum_{i,j,k=1}^n
 \bigl(\E X_iX_jX_k\bigr)^2.
\]
For a fixed $k$, the $(i,j)$ entry of $T_{e_k}$ is
$\E X_iX_jX_k$. Hence
\[
 \sum_{i,j,k=1}^n(\E X_iX_jX_k)^2
 =\sum_{k=1}^n\|T_{e_k}\|_{\HS}^2.
\]
Lemma~\ref{lem:Tu} gives $\|T_{e_k}\|_{\HS}^2\le8$ for every $k$.
Summing over $k$ gives the last inequality in \eqref{eq:Z3}.
\end{proof}

\section{Proof of the directional fluctuation bound}\label{sec:directional-proof}

\subsection{Good and bad events}

We use the following terminology throughout this section.  For
$\zeta\in\{Z,Z^\circ\}$ and for constants to be chosen below, set
\begin{equation}\label{eq:good-event-general}
 \mathcal G_\zeta
 :=\{\alpha_0 n\le a,b\le\beta_0 n,\ A-2\zeta\ge\alpha_1 n\},
 \qquad
 \mathcal B_\zeta:=\mathcal G_\zeta^c.
\end{equation}
We call $\mathcal G_\zeta$ the good event and $\mathcal B_\zeta$ the bad
event.  On the good event both radii are comparable with $\sqrt n$ and
$|X-Y|^2=A-2Z$, or its spherical analogue $A-2Z^\circ$, is bounded below
by a fixed multiple of $n$.  These are exactly the conditions needed to
bound the denominators and the fourth derivative of the kernel.  Every use
of the words ``good event'' or ``bad event'' below refers to
\eqref{eq:good-event-general}.

\begin{lemma}\label{lem:truncation}
There are universal constants
$0<\alpha_0<1<\beta_0$, $0<\alpha_1<1$, and $c,C>0$ such that
\begin{equation}\label{eq:bad-prob}
 \Pp(\mathcal B_Z)+\Pp(\mathcal B_{Z^\circ})
 \le Ce^{-c\sqrt n}.
\end{equation}
Moreover, for every fixed integer $k\ge0$,
\begin{align}
 &\E\bigl[(1+a+b)^k\mathbf1_{\mathcal B_Z}\bigr]
 +\E\bigl[(1+a+b)^k\mathbf1_{\mathcal B_{Z^\circ}}\bigr]
 \notag\\
 &\hspace{5cm}\le C_k n^k e^{-c_k\sqrt n}.
 \label{eq:bad-weighted}
\end{align}
\end{lemma}

\begin{proof}
Choose fixed constants $0<\alpha_0<1<\beta_0$.  Theorem~\ref{thm:GM}
with fixed values of $t$ gives
\[
 \Pp\{a\notin[\alpha_0n,\beta_0n]\}
 +\Pp\{b\notin[\alpha_0n,\beta_0n]\}
 \le Ce^{-c\sqrt n}.
\]
By the preceding subsection, the vector $\widetilde X$ fixed in
\eqref{eq:upper-notation} is centered, isotropic, and log-concave.  Choose
$0<\alpha_1<1$ so small that $\sqrt{\alpha_1/2}<1/2$.  If
$|X-Y|^2<\alpha_1n$, then
\[
 \bigl||\widetilde X|-\sqrt n\bigr|>\frac12\sqrt n.
\]
Apply Theorem~\ref{thm:GM} to $\widetilde X$ with $t=1/2$.  The preceding inclusion
then gives
\[
 \Pp\{|X-Y|^2<\alpha_1 n\}
 \le \Pp\left\{\bigl||\widetilde X|-\sqrt n\bigr|>\frac12\sqrt n\right\}
 \le Ce^{-c\sqrt n}.
\]

For the spherical pair, condition on radii satisfying
$a,b\ge\alpha_0 n$. If $A-2Z^\circ<\alpha_1n$, then
\[
 \ip{U}{V}
 >\frac{A-\alpha_1n}{2\sqrt{ab}}
 \ge1-\frac{\alpha_1}{2\alpha_0}.
\]
Set $\delta=\alpha_1/(2\alpha_0)$ and choose $\alpha_1$ so that
$0<\delta<1/2$. Fix $V=w$ and choose an orthogonal matrix $O$ with
$Oe_1=w$. Since $O^TU$ is uniform on $S^{n-1}$,
\[
 \ip{U}{w}=\ip{O^TU}{e_1}
\]
has the same law as $U_1$. Formula \eqref{eq:U1-density} gives
\begin{align*}
 \Pp\{U_1>1-\delta\}
 &=c_n\int_{1-\delta}^1(1-u^2)^{(n-3)/2}\,\dd u\\
 &\le c_n\delta(2\delta)^{(n-3)/2}
 \le Ce^{-cn},
\end{align*}
To bound $c_n$ without an asymptotic formula, integrate
\eqref{eq:U1-density} over $|u|\le n^{-1/2}$.  On this interval,
$(1-u^2)^{(n-3)/2}\ge(1-1/n)^{(n-3)/2}\ge c$, and therefore
\[
 1\ge c_n\int_{-n^{-1/2}}^{n^{-1/2}}
 (1-u^2)^{(n-3)/2}\,\dd u
 \ge c\,c_n n^{-1/2}.
\]
Thus $c_n\le C\sqrt n$, and the cap probability is at most $Ce^{-cn}$.
This proves \eqref{eq:bad-prob}.

To prove \eqref{eq:bad-weighted}, choose $p>1$ and let $p'$ be its conjugate
exponent. H\"older's inequality gives, for example,
\[
 \E[(1+a+b)^k\mathbf1_{\mathcal B_Z}]
 \le \|(1+a+b)^k\|_{p}\,
 \Pp(\mathcal B_Z)^{1/p'}.
\]
The moment bound \eqref{eq:Rmom} makes the first factor at most $C_kn^k$,
and \eqref{eq:bad-prob} makes the second exponentially small. The proof for $\mathcal B_{Z^\circ}$ is identical.
\end{proof}

The function $S_n$ satisfies the global bound
\begin{equation}\label{eq:formula-growth}
 0\le S_n(a,b,z)
 \le C m_n\sqrt A\left(1+\frac An\right).
\end{equation}
Indeed, Cauchy--Schwarz gives $|z|\le\sqrt{ab}\le A/2$. Also
$D=\ip{x-y}{x+y}$, so
\[
 \frac{D^2}{\sqrt{A-2z}}
 \le \sqrt{A-2z}\,|x+y|^2
 \le 2A\sqrt{A-2z}.
\]
Combining this bound with Lemma~\ref{lem:truncation} shows that the
expectation of $S_n(a,b,\zeta)$ over $\mathcal B_\zeta$ is smaller than
$n^{-N}$ for every fixed $N$, for either $\zeta=Z$ or
$\zeta=Z^\circ$.  We may therefore perform the Taylor expansion on
$\mathcal G_\zeta$, where every denominator is uniformly separated from
zero, and then restore the bad event with an exponentially small error.

Choose a fixed function $\chi\in C_c^\infty((0,\infty))$ such that
$\chi=1$ on $[\alpha_0,\beta_0]$ and
$\operatorname{supp}\chi\subset[\alpha_0/2,2\beta_0]$. Put
\begin{equation}\label{eq:cutoff}
 \chi_n(a,b)=\chi(a/n)\chi(b/n).
\end{equation}
The function $\chi_n$ is used only to make the coefficients smooth on all
of $\R^2$.  By Lemma~\ref{lem:truncation}, $\chi_n(a,b)=1$ except on an
event of probability at most $Ce^{-c\sqrt n}$.

Write the Taylor expansion of $S_n$ in the inner-product variable at
$z=0$ as
\begin{equation}\label{eq:z-expansion}
 S_n(a,b,z)
 =s_0(a,b)+\ell(a,b)z+q(a,b)z^2+c_3(a,b)z^3+\mathcal R_4(a,b,z).
\end{equation}
Differentiating \eqref{eq:formula} at $z=0$ gives
\begin{align}
 \ell(a,b)
 &=\frac{m_n}{\sqrt A}
 \left[-1-\frac{A}{4(n+1)}
 +\frac{D^2}{4(n+1)A}\right],\label{eq:l-formula}\\
 q(a,b)
 &=\frac{m_n}{2A^{3/2}}
 \left[-1-\frac{3A}{4(n+1)}
 +\frac{3D^2}{4(n+1)A}\right],\label{eq:q-formula}\\
 c_3(a,b)
 &=\frac{m_n}{6A^{5/2}}
 \left[-3-\frac{7A}{4(n+1)}
 +\frac{15D^2}{4(n+1)A}\right].\label{eq:c-formula}
\end{align}
The value of $s_0(a,b)=S_n(a,b,0)$ will cancel later and is not needed.  The differentiation is written out in Appendix~A.2.

Define the cut-off coefficients
\begin{equation}\label{eq:cut-coeff}
 L=\chi_n\ell,
 \qquad
 Q=\chi_nq,
 \qquad
 C_3=\chi_nc_3,
 \qquad
 S_0=\chi_ns_0,
\end{equation}
extended smoothly by zero to all of $\R^2$.

\begin{lemma}\label{lem:coeff-bounds}
For $i+j\le3$,
\begin{align}
 |\partial_a^i\partial_b^jL(a,b)|
 &\le C_{i,j}n^{-1-i-j},\label{eq:L-deriv}\\
 |\partial_a^i\partial_b^jQ(a,b)|
 &\le C_{i,j}n^{-2-i-j}.\label{eq:Q-deriv}
\end{align}
Moreover,
\begin{equation}\label{eq:C-deriv}
 |C_3(a,b)|\le Cn^{-3},
 \qquad
 |\nabla C_3(a,b)|\le Cn^{-4}.
\end{equation}
On the support of $\chi_n$, whenever $A-2z\ge\alpha_1 n$,
\begin{equation}\label{eq:R4}
 |\mathcal R_4(a,b,z)|\le Cn^{-4}|z|^4.
\end{equation}
\end{lemma}

\begin{proof}
Put $\alpha=a/n$ and $\beta=b/n$.  On the support of $\chi_n$ the pair
$(\alpha,\beta)$ lies in the fixed compact set
$[\alpha_0/2,2\beta_0]^2$.  Substituting
\[
 A=n(\alpha+\beta),\qquad D=n(\alpha-\beta)
\]
into \eqref{eq:l-formula}--\eqref{eq:c-formula}, and using
$cn^{-1/2}\le m_n\le Cn^{-1/2}$, we can write
\begin{align*}
 L(a,b)&=n^{-1}\mathscr L_n(\alpha,\beta),\\
 Q(a,b)&=n^{-2}\mathscr Q_n(\alpha,\beta),\\
 C_3(a,b)&=n^{-3}\mathscr C_n(\alpha,\beta).
\end{align*}
The three functions on the right are smooth, supported in a fixed compact
set, and all of their derivatives of order at most three are bounded by a
constant independent of $n$.  For example, the factors
$(\alpha+\beta)^{-1/2}$ and $(\alpha+\beta)^{-3/2}$ are bounded with all
needed derivatives because $\alpha+\beta\ge\alpha_0$.  Since
$\partial_a=n^{-1}\partial_\alpha$ and
$\partial_b=n^{-1}\partial_\beta$, the chain rule gives
\eqref{eq:L-deriv}--\eqref{eq:C-deriv}.

For \eqref{eq:R4}, differentiate \eqref{eq:formula} four times in $z$.
Every resulting term is a universal multiple of one of
\begin{align*}
 &m_n(A-2z)^{-7/2},
 &&\frac{m_nA}{n}(A-2z)^{-7/2},\\
 &\frac{m_n}{n}(A-2z)^{-5/2},
 &&\frac{m_nD^2}{n}(A-2z)^{-9/2}.
\end{align*}
If $w$ lies between $0$ and $z$, then
\[
 A-2w\ge \min\{A,A-2z\}\ge c n,
\]
because $A\ge\alpha_0n$ on the support of the cutoff and
$A-2z\ge\alpha_1n$.  Since $|D|\le A\le Cn$, each of the four displayed
terms is at most $Cn^{-4}$ for every such $w$.  Taylor's theorem gives, for
some $w$ between $0$ and $z$,
\[
 \mathcal R_4(a,b,z)=\frac{\partial_z^4S_n(a,b,w)}{4!}z^4.
\]
The derivative bound now gives \eqref{eq:R4}.
\end{proof}

\begin{proposition}\label{prop:truncated-expansion}
For $\zeta=Z$ and for $\zeta=Z^\circ$,
\begin{equation}\label{eq:truncated-expansion}
 \E S_n(a,b,\zeta)
 =\E\bigl[S_0(a,b)+L(a,b)\zeta+Q(a,b)\zeta^2
 +C_3(a,b)\zeta^3\bigr]
 +O(n^{-2}).
\end{equation}
\end{proposition}

\begin{proof}
Fix $\zeta\in\{Z,Z^\circ\}$.  On the explicitly defined event
$\mathcal G_\zeta$, the radii lie in the interval on which
$\chi_n=1$, and $A-2\zeta\ge\alpha_1n$; hence
\eqref{eq:z-expansion} and the remainder estimate \eqref{eq:R4} apply.
On the bad event $\mathcal B_\zeta$, the growth bound
\eqref{eq:formula-growth} gives
\[
 S_n(a,b,\zeta)
 \le Cn^{-1/2}\sqrt{a+b}\left(1+\frac{a+b}{n}\right).
\]
Let
\[
 H_n(a,b):=Cn^{-1/2}\sqrt{a+b}\left(1+\frac{a+b}{n}\right).
\]
For every fixed $p>1$, \eqref{eq:Rmom} gives $\|H_n(a,b)\|_p\le C_p$.
If $p'$ is conjugate to $p$, H\"older's inequality and
\eqref{eq:bad-prob} therefore give
\[
 \E[H_n(a,b)\mathbf1_{\mathcal B_\zeta}]
 \le C_p\Pp(\mathcal B_\zeta)^{1/p'}
 \le C_pe^{-c_p\sqrt n}.
\]
Since an exponential in $\sqrt n$ is bounded by $C_Nn^{-N}$ for every
fixed $N$, the expectation over the bad event is $O(n^{-N})$.  On the support
of $\chi_n$, $|\zeta|\le\sqrt{ab}\le Cn$, and the
coefficient bounds give
\[
 |S_0|+|L\zeta|+|Q\zeta^2|+|C_3\zeta^3|\le C.
\]
Outside the support these four terms vanish.  Their contribution on the bad
event is therefore also exponentially small.

For the Taylor remainder, \eqref{eq:R4} and \eqref{eq:moment-ledger} give
\[
 \E|\mathcal R_4(a,b,Z)|
 \le Cn^{-4}\E Z^4
 \le Cn^{-2}.
\]
For $Z^\circ$, condition first on $a,b$ and use
\eqref{eq:circ-fourth}.  Independence of $a$ and $b$ gives
\[
 \E(Z^\circ)^4
 =\frac{3}{n(n+2)}(\E a^2)^2.
\]
The bound \eqref{eq:Rmom} with $p=4$ gives $\E a^2=\E R_X^4\le Cn^2$.
Hence $\E(Z^\circ)^4\le Cn^2$, and the same $Cn^{-2}$ remainder bound
follows.
\end{proof}

\subsection{The first three Taylor terms}

The radial pair $(a,b)$ has exactly the same distribution in the original
and the two formulas. Hence the $S_0$ terms in
Proposition~\ref{prop:truncated-expansion} cancel. It remains to treat the
next three orders.

Set
\begin{equation}\label{eq:deltas}
 \delta_a=a-n=\Delta_X,
 \qquad
 \delta_b=b-n=\Delta_Y.
\end{equation}

\subsubsection{The linear term}

\begin{proposition}\label{prop:linear}
One has
\begin{equation}\label{eq:linear-bound}
 \left|
 \E[L(a,b)Z^\circ]-\E[L(a,b)Z]
 \right|
 \le\frac C{n^2}.
\end{equation}
\end{proposition}

\begin{proof}
By \eqref{eq:circ-moments},
\begin{equation}\label{eq:linear-circ-zero}
 \E[L(a,b)Z^\circ]=0.
\end{equation}
Let $P_L$ be the second-order Taylor polynomial of $L$ at $(n,n)$.
By \eqref{eq:L-deriv},
\begin{equation}\label{eq:L-rem}
 |L(a,b)-P_L(a,b)|
 \le Cn^{-4}(|\delta_a|+|\delta_b|)^3.
\end{equation}
Using H\"older's inequality and \eqref{eq:moment-ledger},
\begin{align}
 \E\bigl[|L-P_L|\,|Z|\bigr]
 &\le Cn^{-4}
 \norm{|\delta_a|+|\delta_b|}_6^3\norm Z_2
 \le\frac C{n^2}.
 \label{eq:L-rem-expect}
\end{align}

Write
\[
 P_L(a,b)=\ell_0+\ell_a\delta_a+\ell_b\delta_b
 +\frac12\ell_{aa}\delta_a^2
 +\ell_{ab}\delta_a\delta_b
 +\frac12\ell_{bb}\delta_b^2.
\]
For every integrable function $f$ of $a$, independence and $\E Y=0$ give
\begin{equation}\label{eq:pure-linear-zero}
 \E[f(a)Z]
 =\E\!\left[f(a)\ip{X}{Y}\right]
 =\E[f(a)X]\mathbin{\cdot}\E Y=0.
\end{equation}
The same argument, with the roles of $X$ and $Y$ interchanged, applies to
functions of $b$.  Hence the constant term, both linear terms, and the two
pure quadratic terms in $P_L$ vanish after multiplication by $Z$.  The only
term left is $\ell_{ab}\delta_a\delta_b Z$.  By
\eqref{eq:L-deriv}, $|\ell_{ab}|\le Cn^{-3}$, while independence gives
\begin{equation}\label{eq:linear-cross}
 \E[\delta_a\delta_bZ]
 =\sum_{i=1}^n\bigl(\E[\Delta_X X_i]\bigr)^2
 =|v|^2.
\end{equation}
Lemma~\ref{lem:Tu} gives $|v|^2\le8n$. Together with
\eqref{eq:L-rem-expect}, this proves \eqref{eq:linear-bound}.
\end{proof}

\subsubsection{The quadratic term}

\begin{proposition}\label{prop:quadratic}
One has
\begin{equation}\label{eq:quadratic-bound}
 \left|
 \E[Q(a,b)(Z^\circ)^2]-\E[Q(a,b)Z^2]
 \right|
 \le\frac C{n^2}.
\end{equation}
\end{proposition}

\begin{proof}
Let $P_Q$ be the second-order Taylor polynomial of $Q$ at $(n,n)$.
By \eqref{eq:Q-deriv},
\begin{equation}\label{eq:Q-rem}
 |Q(a,b)-P_Q(a,b)|
 \le Cn^{-5}(|\delta_a|+|\delta_b|)^3.
\end{equation}
For the original inner product, H\"older's inequality gives
\begin{align*}
 \E\bigl[|Q-P_Q|Z^2\bigr]
 &\le Cn^{-5}
 \norm{|\delta_a|+|\delta_b|}_6^3\norm Z_4^2.
\end{align*}
The fixed-moment estimates in \eqref{eq:moment-ledger} mean, in the present
calculation, that
\[
 \norm{|\delta_a|+|\delta_b|}_6
 \le \norm{\Delta_X}_6+\norm{\Delta_Y}_6
 \le C\sqrt n,
 \qquad
 \norm Z_4\le C\sqrt n.
\]
Consequently the last display is at most
$Cn^{-5}n^{3/2}n=Cn^{-5/2}$.

For $Z^\circ$, first condition on $a,b$ and use
$\E[(Z^\circ)^2\mid a,b]=ab/n$ from \eqref{eq:circ-moments}.  Then
\begin{align*}
 \E\bigl[|Q-P_Q|(Z^\circ)^2\bigr]
 &\le \frac{C}{n^6}
 \E\bigl[(|\delta_a|+|\delta_b|)^3ab\bigr]\\
 &\le \frac{C}{n^6}
 \norm{|\delta_a|+|\delta_b|}_6^3
 \norm a_4\norm b_4.
\end{align*}
Here $\norm a_4=\norm{R_X}_8^2\le Cn$ and the same bound holds for $b$
by \eqref{eq:Rmom}.  Thus this expectation is also bounded by
$Cn^{-6}n^{3/2}n^2=Cn^{-5/2}$.  Since $n^{-5/2}\le n^{-2}$, the Taylor
remainder contributes at most $Cn^{-2}$ in both cases.

If $f$ is any integrable function of $a$, isotropy and independence give
\begin{equation}\label{eq:pure-quadratic-original}
 \E[f(a)Z^2]
 =\E\left[f(a)X^T(\E YY^T)X\right]
 =\E[f(a)a].
\end{equation}
On the other hand, by \eqref{eq:circ-moments},
\begin{equation}\label{eq:pure-quadratic-circ}
 \E[f(a)(Z^\circ)^2]
 =\E\left[f(a)\frac{ab}{n}\right]
 =\E[f(a)a],
\end{equation}
because $\E b=n$. Write the polynomial explicitly as
\[
 P_Q(a,b)=q_0+q_a\delta_a+q_b\delta_b
 +\frac12q_{aa}\delta_a^2+q_{ab}\delta_a\delta_b
 +\frac12q_{bb}\delta_b^2,
\]
where the coefficients are the corresponding derivatives of $Q$ at
$(n,n)$.  The terms $q_0$, $q_a\delta_a$, and
$q_{aa}\delta_a^2/2$ are functions of $a$ alone, so their expectations
against $Z^2$ and $(Z^\circ)^2$ agree by
\eqref{eq:pure-quadratic-original}--\eqref{eq:pure-quadratic-circ}.  The
same argument applies to the three terms depending only on $b$.  Hence the
difference of the two expectations contains only the term
$q_{ab}\delta_a\delta_b$.

For the original pair,
\begin{equation}\label{eq:quadratic-cross-original}
 \E[\delta_a\delta_bZ^2]
 =\sum_{i,j=1}^n
 \bigl(\E[\Delta_X X_iX_j]\bigr)^2
 =\norm B_{\HS}^2.
\end{equation}
For the pair with independent uniform directions,
\begin{align}
 \E[\delta_a\delta_b(Z^\circ)^2]
 &=\frac1n\E[\delta_a a]\E[\delta_b b]
 \notag\\
 &=\frac{(\E\Delta_X^2)^2}{n}.
 \label{eq:quadratic-cross-circ}
\end{align}
By \eqref{eq:delta-L2} and Lemma~\ref{lem:B}, both quantities in
\eqref{eq:quadratic-cross-original}--\eqref{eq:quadratic-cross-circ} are
$O(n)$. The mixed Taylor coefficient is $O(n^{-4})$ by
\eqref{eq:Q-deriv}. This proves \eqref{eq:quadratic-bound}.
\end{proof}

\subsubsection{The cubic term}

\begin{proposition}\label{prop:cubic}
One has
\begin{equation}\label{eq:cubic-bound}
 \left|
 \E[C_3(a,b)(Z^\circ)^3]-\E[C_3(a,b)Z^3]
 \right|
 \le\frac C{n^2}.
\end{equation}
\end{proposition}

\begin{proof}
The term with independent uniform directions is zero by \eqref{eq:circ-moments}. Put
$c_*=C_3(n,n)=c_3(n,n)$. By \eqref{eq:C-deriv},
\begin{equation}\label{eq:c-star}
 |c_*|\le Cn^{-3},
 \qquad
 |C_3(a,b)-c_*|
 \le Cn^{-4}(|\delta_a|+|\delta_b|).
\end{equation}
The second bound follows from the fundamental theorem of calculus along the
line segment from $(n,n)$ to $(a,b)$.  Since $C_3$ was extended smoothly by
zero and \eqref{eq:C-deriv} is a global bound,
\begin{align*}
 C_3(a,b)-C_3(n,n)
 &=\int_0^1 \nabla C_3(n+s\delta_a,n+s\delta_b)
       \mathbin{\cdot}(\delta_a,\delta_b)\,\dd s.
\end{align*}
Taking absolute values and using $|\nabla C_3|\le Cn^{-4}$ gives
\[
 |C_3(a,b)-c_*|
 \le Cn^{-4}(|\delta_a|+|\delta_b|)
\]
for every $a,b\ge0$.

By Lemma~\ref{lem:Z3},
\begin{equation}\label{eq:cubic-constant}
 |c_*\E Z^3|
 \le Cn^{-3}\cdot8n
 \le\frac C{n^2}.
\end{equation}
For the variable coefficient, Cauchy--Schwarz and
\eqref{eq:moment-ledger} give
\begin{align}
 \E\bigl[|C_3-c_*|\,|Z|^3\bigr]
 &\le Cn^{-4}
 \norm{|\delta_a|+|\delta_b|}_2\norm Z_6^3
 \notag\\
 &\le Cn^{-4}\cdot\sqrt n\cdot n^{3/2}
 \le\frac C{n^2}.
 \label{eq:cubic-variable}
\end{align}
Combining \eqref{eq:cubic-constant} and \eqref{eq:cubic-variable} proves
the proposition.
\end{proof}

\subsection{Completion of the upper bound}

\begin{proposition}\label{prop:fluctuation}
One has
\begin{equation}\label{eq:cdf-identity-final}
 \E_\Theta\cJ(\mu_\Theta,\bar\mu)
 \le\frac C{n^2},
\end{equation}
and hence
\begin{equation}\label{eq:fluctuation}
 \E_\Theta W_1(\mu_\Theta,\bar\mu)
 \le\frac Cn.
\end{equation}
\end{proposition}

\begin{proof}
Apply Proposition~\ref{prop:truncated-expansion} to $Z$ and $Z^\circ$.
The $S_0$ terms cancel because the radial pair $(a,b)$ is identical in the
two expectations. Propositions~\ref{prop:linear},
\ref{prop:quadratic}, and \ref{prop:cubic}, together with the
$O(n^{-2})$ Taylor remainders, yield
\[
 \left|
 \E S_n(a,b,Z^\circ)-\E S_n(a,b,Z)
 \right|
 \le\frac C{n^2}.
\]
Lemma~\ref{lem:cdf-comparison} states that
\[
 \cE_n
 =\frac12\left(\E S_n(a,b,Z^\circ)-\E S_n(a,b,Z)\right).
\]
Since $\cE_n\ge0$, the preceding absolute-value estimate implies
\eqref{eq:cdf-identity-final}.  Lemma~\ref{lem:CDF-reduction} gives, for every
$\theta$,
\[
 W_1(\mu_\theta,\bar\mu)
 \le\sqrt\pi\,\cJ(\mu_\theta,\bar\mu)^{1/2}.
\]
Average this inequality over $\theta$. Cauchy--Schwarz gives
\begin{align*}
 \E_\Theta\cJ(\mu_\Theta,\bar\mu)^{1/2}
 &\le (\E_\Theta 1^2)^{1/2}
 \left(\E_\Theta\cJ(\mu_\Theta,\bar\mu)\right)^{1/2}\\
 &=\left(\E_\Theta\cJ(\mu_\Theta,\bar\mu)\right)^{1/2}.
\end{align*}
Using \eqref{eq:cdf-identity-final} proves
\[
 \E_\Theta W_1(\mu_\Theta,\bar\mu)
 \le\sqrt{\pi\cE_n}
 \le\frac Cn.
\]
\end{proof}

\begin{proof}[Proof of Theorem~\ref{thm:main}]
For every direction $\theta$, the triangle inequality gives
\[
 W_1(\mu_\theta,\gamma_1)
 \le W_1(\mu_\theta,\bar\mu)+W_1(\bar\mu,\gamma_1).
\]
Average over $\theta$ and use Proposition~\ref{prop:fluctuation} and
Proposition~\ref{prop:averaged}. This proves \eqref{eq:main} for all
sufficiently large $n$.

For bounded $n$, let $\delta_0$ be the point mass at zero.  The triangle
inequality gives
\[
 W_1(\mu_\theta,\gamma_1)
 \le W_1(\mu_\theta,\delta_0)+W_1(\delta_0,\gamma_1)
 =\E|\ip{X}{\theta}|+\E|G|.
\]
By Cauchy--Schwarz and isotropy,
$\E|\ip{X}{\theta}|\le(\E\ip{X}{\theta}^2)^{1/2}=1$, and
$\E|G|<1$.  Hence the last display is at most two.
Increasing the universal constant completes all remaining dimensions.
\end{proof}

\appendix
\section{Algebraic details}

This appendix repeats the main algebraic calculations in a compact form.
The displayed formulas show how every power of $n$ used in the proof arises.

\subsection{The spherical matrix moment}

Let $e=e_1$. Sign changes of the coordinates make every off-diagonal entry
zero, and rotations of the last $n-1$ coordinates make their diagonal entries
equal. Thus
\[
 \E[|\Theta_1|\Theta\Theta^T]
 =\operatorname{diag}(\lambda_\parallel,
 \lambda_\perp,\ldots,\lambda_\perp).
\]
Formula \eqref{eq:absolute-sphere-moment}, first with $r=1$ and then with
$r=3$, gives
\[
 \lambda_\parallel=\E|\Theta_1|^3=\frac{2m_n}{n+1}.
\]
Taking the trace gives
$\lambda_\parallel+(n-1)\lambda_\perp=m_n$, hence
$\lambda_\perp=m_n/(n+1)$. This proves
\eqref{eq:sphere-matrix} with no asymptotic approximation.

\subsection{The first three coefficients in the spherical formula}

Remove the common factor $m_n$ and write
\[
 \mathscr S(z)
 =r(z)+\frac{(5A+2z)r(z)}{12(n+1)}
 +\frac{D^2}{4(n+1)r(z)},
 \qquad
 r(z)=(A-2z)^{1/2}.
\]
Using
\[
 r'(z)=-r^{-1},
 \quad
 r''(z)=-r^{-3},
 \quad
 r'''(z)=-3r^{-5},
\]
and the corresponding derivatives of $r^{-1}$, differentiation at
$z=0$ yields
\[
 \frac{\mathscr S'(0)}{1!}
 =\frac1{\sqrt A}
 \left[-1-\frac A{4(n+1)}+\frac{D^2}{4(n+1)A}\right],
\]
\[
 \frac{\mathscr S''(0)}{2!}
 =\frac1{2A^{3/2}}
 \left[-1-\frac{3A}{4(n+1)}
 +\frac{3D^2}{4(n+1)A}\right],
\]
and
\[
 \frac{\mathscr S'''(0)}{3!}
 =\frac1{6A^{5/2}}
 \left[-3-\frac{7A}{4(n+1)}
 +\frac{15D^2}{4(n+1)A}\right].
\]
These are exactly \eqref{eq:l-formula}--\eqref{eq:c-formula}.

\subsection{Checking the powers of n}

On the cutoff support, $A\asymp n$, $|D|\le A$, and $m_n\asymp n^{-1/2}$.
Therefore
\[
 L=O(n^{-1}),
 \qquad
 Q=O(n^{-2}),
 \qquad
 C_3=O(n^{-3}).
\]
Each derivative in $a$ or $b$ introduces one additional factor $n^{-1}$.
The mixed second-order Taylor coefficients are consequently
\[
 [\delta_a\delta_b]P_L=O(n^{-3}),
 \qquad
 [\delta_a\delta_b]P_Q=O(n^{-4}).
\]
The estimates proved in Lemmas~\ref{lem:Tu}, \ref{lem:B}, and \ref{lem:Z3} give
\[
 |v|^2=O(n),
 \qquad
 \norm B_{\HS}^2=O(n),
 \qquad
 \frac{(\E\Delta_X^2)^2}{n}=O(n),
 \qquad
 \E Z^3=O(n).
\]
Thus every term is $O(n^{-2})$ before the final square root in
Lemma~\ref{lem:CDF-reduction}.

\section*{Acknowledgments}

The author thanks Professor Hanchao Wang for helpful discussions and
guidance.


\end{document}